\documentclass[letterpaper, 10 pt]{IEEEtran}  

\IEEEoverridecommandlockouts                              

\usepackage{mathtools}

\usepackage{blkarray}
\usepackage{amsthm}
\usepackage{amssymb}
\usepackage{amsmath}
\usepackage{multirow}

\usepackage{graphicx,epsfig,color,amsfonts,subfigure}
\usepackage{version,xspace}
\usepackage{tikz}
\usepackage{centernot}

\usepackage{url,doi}

\usepackage{environ}
\usepackage{bbm}
\usepackage{bbold}
\usepackage{bm}

\usepackage{algorithm}
\usepackage{algorithmic}
\usepackage{blkarray}
\usetikzlibrary{automata,positioning}

\usepackage{makecell}

\DeclareMathOperator*{\argmax}{argmax}
\DeclareMathOperator*{\argmin}{argmin}

\usetikzlibrary{calc}

\def\real{\mathbb{R}}

\def\vspecdots{\vbox{\baselineskip=2pt \lineskiplimit=0pt
		\kern2pt \hbox{.}\hbox{.}\hbox{.}}}
\newcommand{\until}[1]{\{1,\dots, #1\}}

\newcommand{\kron}{\operatorname{\otimes}}

\renewcommand{\theenumi}{(\roman{enumi})}

\DeclareMathOperator{\vecz}{vec}

\newcommand\oprocendsymbol{\hbox{$\square$}}
\newcommand\oprocend{\relax\ifmmode\else\unskip\hfill\fi\oprocendsymbol}

\DeclareSymbolFont{bbold}{U}{bbold}{m}{n}
\DeclareSymbolFontAlphabet{\mathbbold}{bbold}

\newtheorem{theorem}{Theorem}
\newtheorem{lemma}[theorem]{Lemma}

\newtheorem{remark}[theorem]{Remark}

\newtheorem{problem}{Problem}
\newtheorem{conjecture}{Conjecture}

\title{\LARGE \bf
Stochastic Strategies for Robotic Surveillance \\as Stackelberg Games
}

\author{Xiaoming Duan, Dario Paccagnan, and Francesco Bullo,~\IEEEmembership{Fellow,~IEEE}
\thanks{This work has been supported in part by Air Force Office of Scientific
  Research award FA9550-15-1-0138, and in part by the Swiss National Science
  Foundation under grant number P2EZP2-181618.}
\thanks{Xiaoming Duan, Dario Paccagnan, and Francesco Bullo are with the Mechanical
	Engineering Department and the Center of Control, Dynamical
	Systems and Computation, UC Santa Barbara, CA 93106-5070, USA.
	{\tt\small \{xiaomingduan,dariop,bullo\}@ucsb.edu}}%
}

\begin{document}

\maketitle
\thispagestyle{empty}
\pagestyle{empty}

\begin{abstract}
This paper studies a stochastic robotic surveillance problem where a mobile robot moves randomly on a graph to capture a potential intruder that strategically attacks a location on the graph. The intruder is assumed to be omniscient: it knows the current location of the mobile agent and can learn the surveillance strategy. The goal for the mobile robot is to design a stochastic strategy so as to maximize the probability of capturing the intruder. We model the strategic interactions between the surveillance robot and the intruder as a Stackelberg game, and optimal and suboptimal Markov chain based surveillance strategies in star, complete and line graphs are studied. We first derive a universal upper bound on the capture probability, i.e., the performance limit for the surveillance agent. We show that this upper bound is tight in the complete graph and further provide suboptimality guarantees for a natural design. For the star and line graphs, we first characterize dominant strategies for the surveillance agent and the intruder. Then, we rigorously prove the optimal strategy for the surveillance agent.
\end{abstract}

\begin{IEEEkeywords}
Stochastic robotic surveillance, Markov chains, Stackelberg game, capture probabilities
\end{IEEEkeywords}

\section{Introduction}
\subsubsection*{Problem description and motivation}
In a prototypical robotic surveillance scenario, mobiles robots patrol and
move among locations in an environment (usually modeled by a graph) with
the goal of capturing potential intruders; see Fig.~\ref{fig:surveillance}
for an illustration. Here, we consider a similar setup, where the
patrolling robot is, in addition, facing an omniscient intruder. We model
the strategic interactions between the mobile robot and the intruder as a
Stackelberg game, where the optimal strategy for the surveillance agent is
constructed under the assumption that the intruder acts optimally against
her. This formulation captures the worst-case scenario for the surveillance
agent when playing against the strongest possible opponent. The
corresponding Stackelberg solution is meaningful and practical when little
or no information is known about the intruder. Similar models have appeared
in \cite{NB-NG-FA:12} and \cite{ABA-SLS:16}, where heuristic algorithms
without performance guarantees are provided. Instead, we analyze the
problem from a mathematical perspective and obtain provably optimal or
suboptimal solutions by considering three topologies: star, line and
complete graph.

\begin{figure}[http]
	\centering
	\includegraphics[scale = 0.6]{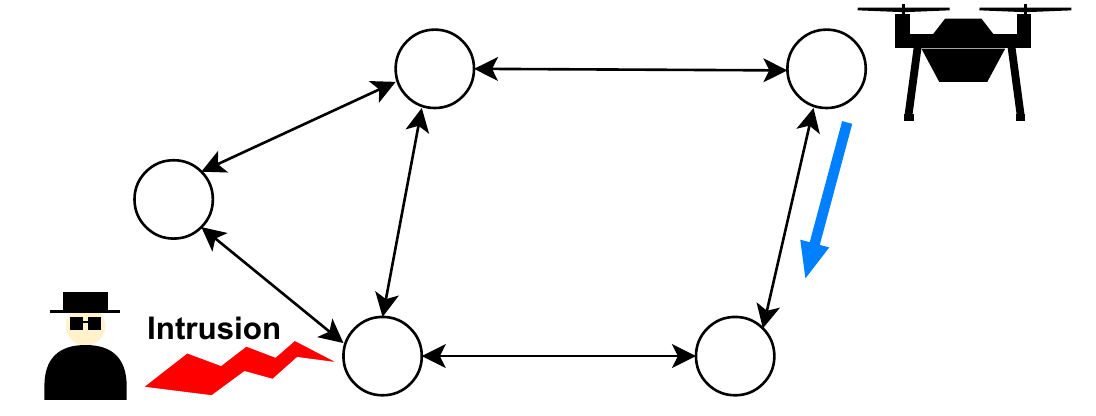}
	\caption{A surveillance scenario where a mobile robot patrols a graph with the goal of capturing potential intruders that  attack certain locations.}\label{fig:surveillance}
\end{figure}

\subsubsection*{Related work}
There have been continuing efforts to study robotic surveillance problems
under various settings, formulations, and assumptions. The early work on
designing deterministic surveillance strategies started
in~\cite{YC:04}. Recent deterministic extensions to cases where locations
might have different importance, or to the coordination of multiple robots
can be found in~\cite{YE-NA-GAK:09, fp-af-fb:09v, FP-JWD-FB:11h,
  SA-EF-SLS:14, ABA-SLS-SS:19}.  Unfortunately, deterministic strategies can
be easily learned, and thus exploited in adversarial settings. In this
respect, stochastic surveillance strategies are more appealing in that they
are mostly unpredictable. One common approach to derive stochastic
surveillance strategies is to model the motion of the surveillance agent as
a first-order Markov chain~\cite{JG-JB:05, GC-AS:11,
  NN-AR-JVH-VI:16}. Within this line of work, Patel \emph{et
  al.}~\cite{RP-PA-FB:14b} studied minimum mean hitting time Markov chains
for robotic surveillance with travel times on edges. They formulated a
convex optimization problem by restricting attentions to the class of
reversible Markov chains. More recently, George \emph{et
  al.}~\cite{MG-SJ-FB:17b} and Duan \emph{et al.}~\cite{XD-MG-FB:17o}
studied and quantified unpredictability of Markov chains and designed
maxentropic surveillance strategies.

The aforementioned works do not explicitly describe the behavior of the malicious intruders. On the other hand, when these models are available, they can be leveraged to design improved surveillance strategies. In order to model the interplay between the surveillance agent and the intruder, surveillance problems have also been studied under a game-theoretic lens. Stackelberg security games, where the defenders and intruders are modeled as strategic players with sequential plays, have been successfully applied in various real-world scenarios \cite{AS-FF-BA-CK-MT:18} such as checkpoints placement and patrolling at airports \cite{JP-MJ-JM-FO-CP-MT-CW-PP-SK:08}, coast guard surveillance \cite{ES-BA-RY-MT-CB-JD-BM-GM:12} and wild life protection \cite{FF-PS-MT:15}. In these works, defenders allocate limited resources to a set of targets so as to optimize their objectives. The problems are formulated as matrix games and the topology of the environments are not explicitly taken into account. In~\cite{NB-NG-FA:12}, the authors introduce a patrolling game where a mobile robot moves on a graph to capture potential intruders who choose when and where to attack. Different intruder models were proposed and analyzed in \cite{ABA-SLS:16}, where the optimal Markov chain based stochastic strategy is computed via a pattern search algorithm. The authors in \cite{ABA-SLS:18} consider an intruder with limited observation time and design a strategy that is both hard to learn and hard to attack. An intruder model where the intruder decides where, when and for how long it attacks is studied in \cite{HY-ST-KSL-SL-JG:19}, and it is found that increasing the randomness of the strategy helps reduce the intruder's reward. The authors in \cite{JB-SK:17} study the impact of the graph topology on the maximization of the minimum expected hitting times. Finally, a sophisticated non-Markovian model was studied in  \cite{TB-PH-AK-VR-MA:15} for the case of complete graphs.

While most of the existing works are concerned with the design of heuristics to compute suboptimal strategies without performance guarantees, in this paper, we derive provably optimal/suboptimal strategies. Towards this goal, we adopt the model of~\cite{ABA-SLS:16} and focus on graphs that correspond to prototypical robotic roadmaps.

\subsubsection*{Contributions}
In this paper, we derive provably optimal and suboptimal strategies for surveillance agents in a Stackelberg game setting. We consider three prototypical robotic roadmaps are considered: star, complete and line graphs. Our problem formulation captures the worst-case scenario for the surveillance agent and thus the solution provides performance guarantees also in less pessimistic scenarios. Our main contributions are as follows.
\begin{enumerate}
	\item We derive a universal upper bound for the capture probability, i.e., the maximum achievable  performance for the surveillance agent facing an omniscient intruder;
	\item We show that this upper bound is tight in the case of the complete graph, and further provide suboptimality guarantees for a natural strategy often referred to as a random walk;
	\item We study dominant strategies for both the intruder and the surveillance agent. Leveraging these insights, we obtain optimal strategies in the star and line graphs.
\end{enumerate}

\subsubsection*{Paper organization}
We provide preliminaries on Markov chains and formulate the Stackelberg game problem in Section~\ref{sec:prelim} and Section~\ref{sec:ProblemFormulation}, respectively. An upper bound on the capture probability and a suboptimal solution in the case of complete graph are obtained in Section~\ref{sec:upperboundcomplete}. We study dominant strategies for the players in Section~\ref{sec:dominancestarline}, and optimal strategies for star and line graphs are given in the same section. 

\subsubsection*{Notation}
Let $\real$, $\mathbb{Z}_{\geq0}$, and $\mathbb{Z}_{>0}$ denote the set of real numbers, nonnegative and positive integers, respectively. Let $\mathbb{1}_n$ and $\mathbb{0}_n$ denote column vectors in $\real^n$ with all entries being $1$ and $0$. $I_n\in \real^{n\times n}$ is the identity matrix. $\mathbb{e}_i$ denotes the $i$-th vector in the standard basis, whose dimension will be made clear when it appears. $[S]$ denotes a diagonal matrix with diagonal elements being $S$ if $S$ is a vector, or being the diagonal of $S$ if $S$ is a square matrix. Let $\otimes$ denote the Kronecker product. $\vecz(\cdot)$ is the vectorization operator that converts a matrix into a column vector.

\section{Preliminaries}\label{sec:prelim}
\subsection{Markov chains}
We start by reviewing the basics of discrete-time Markov chains. A finite-state discrete-time homogeneous Markov chain with state space $\until{n}$ is a sequence of random variables $X_k$, $k\in\mathbb{Z}_{\geq0}$, taking values in $\until{n}$ and satisfying the Markov property, i.e., $X_k$ is such that $\mathbb{P}(X_{k+1}=j\,|\,X_{k}=i,\dots,X_{0}=i_0) =\mathbb{P}(X_{k+1}=j\,|\,X_{k}=i)=p_{ij},$ for all $i,j\in\until{n}$ and $k\in\mathbb{Z}_{\geq0}$, where $p_{ij}$ is the transition probability from state $i$ to state $j$, $P=\{p_{ij}\}\in\real^{n\times{n}}$ is the transition matrix satisfying $P\geq 0$, and $P\mathbb{1}_n=\mathbb{1}_n$; see \cite{JGK-JLS:76}, \cite{JRN:97}. A probability distribution $\bm{\pi}\in\real^n$ is \emph{stationary} for a Markov chain with transition matrix $P$ if it satisfies $\bm{\pi}\geq0$, $\bm{\pi}^\top\mathbb{1}_n=1$ and $\bm{\pi}^\top=\bm{\pi}^\top P$. The transition diagram of a Markov chain $P=\{p_{ij}\}\in\real^{n\times{n}}$ is a directed graph (digraph) $\mathcal{G}=(V,\mathcal{E})$ where $V=\{1,\dots,n\}$ and $\mathcal{E}=\{(i,j)\,|\,i,j\in V, p_{ij}>0\}$. A Markov chain is \emph{irreducible} if its transition diagram is strongly connected. 

\subsection{Hitting time of Markov chains}\label{sec:returntime}
In this paper, we consider a strongly connected digraph $\mathcal{G}=(V,\mathcal{E})$, where $V$ denotes the set of $n$ nodes $\until{n}$ and $\mathcal{E}\subset V\times V$ denotes the set of edges. Given the graph $\mathcal{G}=\{V,\mathcal{E}\}$, let $X_k\in \until{n}$ be the value of a Markov chain with transition diagram $\mathcal{G}$ and transition matrix $P$ at time $k \in \mathbb{Z}_{\geq0}$. For any pair of nodes $i,j\in V$, the \emph{first hitting time} from $i$ to $j$, denoted by $T_{ij}$, is the first time the Markov chain hits node $j$ starting from node $i$, that is
\begin{equation*}
T_{ij}=\min\{k\,|\,X_0=i,X_k=j,k\geq1\}.
\end{equation*}
Note that $T_{ij}$ is itself a random variable. Let the $(i,j)$-th element of the \emph{first hitting time probability matrix} $F_k$ denote the probability that the Markov chain hits node $j$ for the first time in exactly $k$ time units starting from node $i$, i.e., $F_k(i,j)=\mathbb{P}(T_{ij}=k)$. It can be shown that the hitting time probabilities $F_k$ for $k\geq1$ satisfy the following recursive matrix equation \cite[Chapter 5, Eq. (2.4)]{EC:13}
\begin{equation}\label{eq:Fk}
F_{k+1}=P(F_k-\textup{diag}(F_k)),
\end{equation}
where $F_1=P$. The vectorized form of~\eqref{eq:Fk} can be written as
\begin{equation}\label{eq:RecursiveInVec}
\vecz(F_{k+1}) =(I_n\otimes P)(I_{n^2} - E)\vecz(F_{k}),
\end{equation}
where $E = \textup{diag}(\vecz(I_n))$. Note that~\eqref{eq:Fk} can be generalized to the case where there are times on the graph $\mathcal{G}$ \cite{XD-MG-FB:17o}.

\section{Problem formulation}\label{sec:ProblemFormulation}
We consider a robotic surveillance problem where a mobile robot moves randomly between locations in a graph to perform surveillance tasks. Specifically, given a Markov chain strategy, the surveillance robot moves from the current node to a neighboring location according to the corresponding Markov chain transition matrix. An intruder attacks an unknown node in the graph by stationing at the given node for a certain period of time. The intruder is captured if the surveillance agent visits that node within the duration of the attack; see Fig.~\ref{fig:surveillance} for a pictorial representation.
 In this work, we study how the mobile robot should move on the graph in order to maximize the probability of capturing the intruder. 
 
We model the strategic interactions between the surveillance robot and the intruder as a two-player Stackelberg game with a leader and a follower. The game proceeds as follows: the leader commits to a strategy first, and then the follower, based on the knowledge of the leader's strategy, selects a strategy that optimizes her rewards. Having prior knowledge of the follower's best-response, the leader commits to a strategy that ultimately maximizes her own objective. In our problem setting, the surveillance agent is the leader who chooses a Markov chain as surveillance strategy. This strategy is observed and learned by the intruder (follower) who then chooses the best time and location to attack so as to minimize the probability of being captured. We describe the intruder and surveillance models in the following subsections.

\subsection{Intruder Model}
An intruder aims to attack a location in the graph while the surveillance agent moves around with the goal of capturing her. The intruder requires $\tau$ units of time to complete the attack at any location in the graph. Once it commits to attacking a location, it stations at the location for that given period of time. We assume that the intruder is omniscient \cite{NB-NG-FA:12,ABA-SLS:16}, i.e., it knows or can learn the strategy (intended as a description of the Markov chain) as well as the current location of the surveillance agent perfectly. Given this information, the intruder decides when and where to attack so that it is least likely to be captured. 

Given a Markov chain strategy for the surveillance agent, the intruder picks a pair of locations $i$ and $j$ so that the probability that the surveillance agent goes from location $i$ to location $j$ within the attack duration is minimized. Then, the intruder attacks location $j$ whenever the surveillance agent is at location $i$. Formally, for a surveillance strategy parameterized by a Markov transition matrix $P$, an optimal strategy for the intruder $(i^*,j^*)$ is given by
\begin{equation}\label{eq:smartestintruder}
(i^*,j^*)\in\argmin_{i,j}\{\mathbb{P}(T_{ij}(P)\leq \tau)\}.
\end{equation}
Note that optimal strategies in~\eqref{eq:smartestintruder} are not unique in general, and the intruder is free to pick any one of them.

\subsection{Surveillance model and problem formulation}
The surveillance agent, knowing that the intruder selects an optimal strategy according to \eqref{eq:smartestintruder}, adopts a Markov chain $P^*$ that maximizes the probability of capturing the intruder, i.e.,
\begin{equation*}
P^*=\argmax_{P}\min_{i,j}\{\mathbb{P}(T_{ij}(P)\leq \tau)\}.
\end{equation*}

In this paper, we are interested in finding an optimal strategy for the surveillance agent when playing against the omniscient intruder just described. That is, we are interested in solving the following optimization problem.
\begin{problem}
\label{prob:optimalstrategy}
Given a strongly connected digraph $\mathcal{G} = (V, \mathcal{E})$ and the attack duration $\tau\in\mathbb{Z}_{>0}$, find a Markov chain that conforms to the graph topology and maximizes the capture probability, i.e., solves the following optimization
	problem:
\begin{equation}\label{eq:originalopt}
\begin{aligned}
& \underset{P\in\mathbb{R}^{n\times n}}{\textup{maximize}}
& &\min_{i,j}\{\mathbb{P}(T_{ij}(P)\leq \tau)\}\\
& \textup{subject to}
&& P\mathbb{1}_n=\mathbb{1}_n,\\
&&& p_{ij}\geq0,\quad\textup{for all }(i,j)\in\mathcal{E},\\
&&& p_{ij}=0,\quad\textup{for all }(i,j)\notin\mathcal{E}.
\end{aligned}
\end{equation}
\end{problem}
We let $\mathbb{V}\in[0,1]$ denote the optimal value of Problem~\ref{prob:optimalstrategy} and refer to it as the value of the game. Note that $\mathbb{V}$ represents the capture probability obtained when the surveillance agent utilizes an optimal strategy against the omniscent intruder. As a consequence, $\mathbb{V}$ lower bounds the capture probability also in less pessimistic circumstances, e.g., in the case of a nonstrategic intruder.
\begin{remark}[Impact of the attack duration]\label{rk:attackduration}
Problem~\ref{prob:optimalstrategy} is interesting only when the attack duration $\tau$ is in an appropriate range for a given graph topology. Specifically, 
\begin{enumerate}
	\item the attack duration should be greater than or equal to the diameter $D$ of graph $\mathcal{G}$, i.e., $\tau\geq D$. Otherwise, the omniscient intruder always succeeds by attacking one end of the graph diameter when the surveillance agent is visiting the other. In this case, $\mathbb{V}=0$ no matter what strategy the surveillance agent uses;
	\item the attack duration should be smaller than the length of any closed path on the graph $\mathcal{G}$ that has the same initial and final vertices and visits all locations at least once (e.g., a Hamiltonian tour of size $n$ if it exists). Otherwise, the surveillance agent does not benefit from using a Markov chain as a randomized strategy, and the capture is guaranteed by following the deterministic closed path.
\end{enumerate}
\end{remark}
We assume hereafter that the attack duration $\tau$ takes a nontrivial value as described in Remark~\ref{rk:attackduration}.
\begin{remark}[Irreducibility of optimal solutions]\label{remark:irreducibility}
No irreducibility constraint is imposed on the Markov chain in Problem~\ref{prob:optimalstrategy}. However, if $\tau$ takes nontrivial values, an optimal Markov chain is necessarily irreducible. As the transition diagrams of any reducible Markov chains is not strongly connected, there must exist a pair of locations $i$ and $j$ such that $\mathbb{P}(T_{ij}\leq\tau)=0$, so that also $\mathbb{V}=0$.
\end{remark}

\begin{remark}[Graph dimension]\label{rk:graphdim}
Without loss of generality, we consider graphs with more than $2$ nodes in the rest of this paper, i.e., $n\geq3$. When $n=2$, the only meaningful case is a complete graph, and the 
optimal solution to Problem~\ref{prob:optimalstrategy} is $P=\frac{1}{2}\mathbb{1}_2\mathbb{1}_2^\top$ with $\mathbb{V}=\frac{1}{2}$ if $\tau=1$, and an irreducible permutation matrix (a Hamiltonian tour) with $\mathbb{V}=1$ if $\tau\geq2$.
\end{remark}

\section{Upper bounds of value of the game and suboptimal solution in complete graphs}\label{sec:upperboundcomplete}
In this section, we first derive a universal upper bound for the value of the game, which does not depend the graph topology. We then consider the case of complete graph and show that the upper bound can be tight. Further, we provide suboptimality guarantees for a a Markov chain whose corresponding transition matrix has identical entries.

\subsection{Upper bound of the value of the game}
We introduce an auxiliary variable $\mu\in\mathbb{R}$ and exploit the iteration~\eqref{eq:Fk} to rewrite the optimization problem~\eqref{eq:originalopt} as follows
\begin{equation}\label{eq:equivalentopt}
\begin{aligned}
& \underset{\mu\in\mathbb{R},P\in\mathbb{R}^{n\times n}}{\textup{maximize}}
& &\mu\\
& \textup{subject to}
&& \mu\mathbb{1}_n\mathbb{1}_n^\top\leq\sum_{k=1}^{\tau} F_k,\\
&&&F_1 = P\\
&&& F_{k+1}=P(F_k-\textup{diag}(F_k)),\quad 1\leq k\leq\tau-1\\
&&& P\mathbb{1}_n=\mathbb{1}_n,\\
&&& p_{ij}\geq0,\quad\textup{for all }(i,j)\in\mathcal{E},\\
&&& p_{ij}=0,\quad\textup{for all }(i,j)\notin\mathcal{E},
\end{aligned}
\end{equation}
where the inequality in the first constraint is element-wise. Clearly, problem~\eqref{eq:equivalentopt} is equivalent to~\eqref{eq:originalopt}, and the optimal value $\mu^*$ is the value of the game. The following theorem shows how to obtain a universal upper bound on $\mu^*$. 

\begin{theorem}[Upper bound for the value of the game]\label{thm:upperbound}
Given a strongly connected digraph $\mathcal{G}=(V,\mathcal{E})$ with $n$ nodes and an attack duration $\tau$ that takes nontrivial values as in Remark~\ref{rk:attackduration}, 
the value of the game satisfies $\mathbb{V}\leq\frac{\tau}{n}$.
\end{theorem} 
\begin{proof}
By Remark~\ref{remark:irreducibility}, 
an optimal solution $P^*$ to problem~\eqref{eq:equivalentopt} is irreducible and thus has a unique stationary distribution $\bm\pi$. We multiply $\bm\pi^\top$ from the left on both sides of \eqref{eq:Fk} and obtain for $k\geq 1$,
\begin{equation}\label{eq:multistationary}
\bm\pi^\top F_{k+1}=\bm\pi^\top (F_k-\textup{diag}(F_k))\leq \bm\pi^\top F_k.
\end{equation}
By using \eqref{eq:multistationary} recursively, we have that for $k\geq1$,
\begin{equation}\label{eq:uppderboundF}
\bm\pi^\top F_k\leq \bm\pi^\top F_1= \bm\pi^\top P^*=\bm\pi^\top.
\end{equation}
Since $(\mu^*,P^*)$ is an optimal solution to
problem~\eqref{eq:equivalentopt}, it satisfies the first constraint and
thus
\begin{equation}\label{eq:firstconst}
\mu^*\mathbb{1}_n\mathbb{1}_n^\top\leq\sum_{k=1}^{\tau} F_k.
\end{equation}
Multiplying $\bm\pi^\top$ from the left on both sides of \eqref{eq:firstconst} and using \eqref{eq:uppderboundF}, we obtain
\begin{equation*}
\mu^*\mathbb{1}_n^\top\leq\sum_{k=1}^{\tau}\bm\pi^\top=\tau\bm\pi^\top.
\end{equation*}
Since $\bm\pi^\top\mathbb{1}_n=1$, we must have $\min\limits_{1\leq i\leq n}\pi_i\leq\frac{1}{n}$. Therefore, 
\begin{equation*}
\mathbb{V}=\mu^*\leq\tau\min_{1\leq i\leq n}\pi_i=\frac{\tau}{n}.
\end{equation*}
\end{proof}

\subsection{Suboptimal solution in complete digraphs}
Next, we show that for complete digraphs the upper bound $\mathbb{V}\le \frac{\tau}{n}$ can be achieved for certain combinations of $n$ and $\tau$.

\begin{lemma}[Optimal solution in special complete digraph]\label{lemma:completeOpt}
Given a complete digraph $\mathcal{G}=(V,\mathcal{E})$ with $n$ nodes and the attack duration $\tau\leq n$, if $\tau$ divides $n$, then an optimal solution $P^*$ is given by
\begin{equation*}
P^*=\Pi_0\kron\frac{\tau}{n}\mathbb{1}_{\frac{n}{\tau}}\mathbb{1}_{\frac{n}{\tau}}^\top,
\end{equation*}
where $\Pi_0\in\mathbb{R}^{\tau \times\tau}$ is any irreducible permutation matrix that represents a Hamiltonian tour in $\mathcal{G}$.
Moreover, the optimal strategy $P^*$ achieves the upper bound of the value of the game.
\end{lemma}
\begin{proof}
By construction, the probability that a surveillance agent with $P^*$ starting from any location $i$ arrives at any location $j$ within $\tau$ time steps is $\frac{\tau}{n}$.
\end{proof}

The previous result holds only when the combination of $n$ and $\tau$ is such that $\tau$ divides $n$. Nevertheless, leveraging Theorem~\ref{thm:upperbound} we are able to provide suboptimality guarantees for the natural choice of $P=\frac{1}{n}\mathbb{1}_n\mathbb{1}_n^\top$.

\begin{lemma}[Constant factor optimality of random walk]\label{thm:complete}
Given a complete digraph $\mathcal{G}=(V,\mathcal{E})$ with $n\geq3$ nodes and the attack duration $\tau$, the random walk strategy $P=\frac{1}{n}\mathbb{1}_n\mathbb{1}_n^\top$ achieves performance within
\begin{equation*}
\frac{n^\tau-(n-1)^\tau}{\tau n^{\tau-1}}\geq 1-\frac{1}{e}
\end{equation*}
of optimality, where $e$ is Euler's number.
\end{lemma}
\begin{proof}
First, note that the capture probability for the random walk surveillance policy $P=\frac{1}{n}\mathbb{1}_n\mathbb{1}_n^\top$ is $1-(1-\frac{1}{n})^\tau$. Therefore, the random walk achieves performance within
\begin{equation*}
f(n,\tau)=\frac{n^\tau-(n-1)^\tau}{\tau n^{\tau-1}}
\end{equation*}
of optimality.
Let $n$ be fixed, and $g(\tau)=\frac{n^\tau-(n-1)^\tau}{\tau n^{\tau-1}}$. We relax $\tau$ to be a continuous variable and take derivative of $g(\tau)$ as
\begin{align*}
\frac{dg(\tau)}{d\tau}&=-\frac{n}{\tau^2}+\frac{n}{\tau^2}\big(1-\frac{1}{n}\big)^\tau-\frac{n}{\tau}\big(1-\frac{1}{n}\big)^\tau\log\big(1-\frac{1}{n}\big)\\
&=\frac{n}{\tau^2}\big((1-\tau\log(1-\frac{1}{n}))(1-\frac{1}{n})^\tau-1\big)\\
&\leq\frac{n}{\tau^2}\big((1-\tau(1-\frac{1}{1-\frac{1}{n}}))(1-\frac{1}{n})^\tau-1\big)\\
&=\frac{n}{\tau^2}\big((1+\frac{\tau}{n-1})(1-\frac{1}{n})^\tau-1\big)\leq0,
\end{align*}
where we used the upper bound $\log x\geq 1-\frac{1}{x}$ in the first inequality, and $\tau\leq n-1$ and $(1-\frac{1}{n})^\tau\leq\frac{1}{2}$ in the second inequality. Therefore, for fixed $n$, $f(n,\tau)$ is a decreasing function in $\tau$ and
\begin{equation*}
f(n,\tau)\geq f(n,n-1)=\frac{n}{n-1}-(1-\frac{1}{n})^{n-2}\triangleq h(n).
\end{equation*}
We relax $n$ to be a continuous variable in this case and compute the derivative of $h(n)$ as
\begin{align*}
\frac{dh(n)}{dn}&=-\frac{1}{(n-1)^2}+(1-\frac{1}{n})^{n-2}\big(\log\frac{n}{n-1}-\frac{n-2}{n^2-n}\big)\\
&\leq-\frac{1}{(n-1)^2}+(1-\frac{1}{n})^{n-2}(\frac{1}{n-1}-\frac{n-2}{n^2-n})\\
&=-\frac{1}{(n-1)^2}+(1-\frac{1}{n})^{n-1}\frac{2}{(n-1)^2}\\
&\leq-\frac{1}{(n-1)^2}+(1-\frac{1}{3})^{2}\frac{2}{(n-1)^2}\leq0,
%
%
%
\end{align*}
where we used the upper bound $\log(1+x)\leq x$ and the proved that $(1-\frac{1}{n})^{n-1}$ is decreasing as $n$ increases. Therefore, we have that $h(n)$ is a decreasing function. In summary, we have
\begin{multline*}
f(n,\tau)\geq f(n,n-1)\\
\geq\lim_{n\rightarrow\infty}h(n)=\lim_{n\rightarrow\infty}\big(\frac{n}{n-1}-(1-\frac{1}{n})^{n-2}\big)=1-\frac{1}{e}.
\end{multline*}
\end{proof}

\section{Strategy dominance and optimal strategies in star and line graphs}\label{sec:dominancestarline}
In this section, we first obtain dominated strategies for the intruder as well as the dominant strategies for the surveillance agent. Leveraging these result, we derive optimal strategies for the surveillance agent in star and line graphs.
\subsection{Dominated strategies for the omniscient intruder}
In this subsection, we present two lemmas characterizing dominated strategies for the omniscient intruder, i.e., strategies that intruder will never choose as they lead to a higher probability of being captured.
\begin{lemma}[Dominated strategy for the intruder]\label{lemma:dominance_int}
Given a strongly connected digraph $\mathcal{G}=(V,\mathcal{E})$ and an irreducible Markov chain strategy $P$ for the surveillance agent, attacking node $j\in V$ when the surveillance agent is at node $i\in V$, $i\neq j$, is not optimal for the omniscient intruder if:
\begin{enumerate}
\item\label{item:pre} there exists a node $k$ such that any path from node $k$ to node $j$ contains node $i$; or
\item\label{item:post} there exists a node $k$ such that any path from node $i$ to node $k$ contains node $j$.
\end{enumerate}
Moreover, in case \ref{item:pre} (resp. in case \ref{item:post}) , attacking node $j$ when the surveillance agent is at node $k$  (resp. node $i$) is a better strategy for the omniscient intruder.
\end{lemma}
\begin{proof}
The probability that the surveillance agent visiting node $i$ captures the intruder attacking node $j$ is
\begin{equation*}
\mathbb{P}(T_{ij}\leq \tau)=\sum_{ t=1}^{\tau} \mathbb{P}(T_{ij}=  t).
\end{equation*}

Regarding~\ref{item:pre}, we need to show that
\begin{equation*}
\mathbb{P}(T_{kj}\leq \tau)\leq\mathbb{P}(T_{ij}\leq \tau).
\end{equation*}
Since any path from node $k$ to node $j$ contains node $i$, by definition
of probability, and the memoryless property of Markov chains, we have
\begin{align*}
\mathbb{P}(T_{kj}\leq\tau)&=\sum_{ t=1}^{\tau-1}\mathbb{P}(T_{ki}=  t)\mathbb{P}(T_{ij}\leq \tau-  t)\\
&\leq \max_{1\leq  t\leq\tau-1}\mathbb{P}(T_{ij}\leq \tau-  t)\\
&= \mathbb{P}(T_{ij}\leq \tau- 1)\\
&\leq \mathbb{P}(T_{ij}\leq \tau),
\end{align*}
where we used that $\mathbb{P}(T_{ij}\leq \tau- t)$ is  decreasing  with $t$.

The proof for~\ref{item:post} follows a similar argument as in~\ref{item:pre}.
\end{proof}
\begin{lemma}[Dominated strategy on leaf nodes]\label{coro:intruder_dominance}
Given a strongly connected digraph $\mathcal{G}=(V,\mathcal{E})$ with $n\geq3$ nodes and an irreducible Markov chain strategy $P$ for the surveillance agent, if node $i\in V$ is a leaf node in $\mathcal{G}$, then attacking node $i$ when the surveillance agent just leaves node $i$ is not optimal for the omniscient intruder.
\end{lemma}
\begin{proof}
The case of $\tau=1$ is uninteresting here because the surveillance agent fails with probability $1$ if the intruder attacks the leaf node $i$ when the surveillance agent visits other nodes than node $i$ and its neighbor. Therefore, we consider $\tau\geq2$ in the following. Let node $j$ be the neighbor node of the leaf node $i$, then 
\begin{equation}\label{eq:self}
\mathbb{P}(T_{ii}\leq\tau)=p_{ii}+(1-p_{ii})\mathbb{P}(T_{ji}\leq\tau-1).
\end{equation}
Moreover, for $k\in V$, $k\neq i$ and $k\neq j$,
\begin{align}\label{eq:ineq}
\begin{split}
\mathbb{P}(T_{ki}\leq\tau)&=\sum_{ t=1}^{\tau-1}\mathbb{P}(T_{kj}= t)\mathbb{P}(T_{ji}\leq \tau- t)\\
&\leq \max_{1\leq  t\leq \tau-1}\mathbb{P}(T_{ji}\leq \tau- t)\\
&= \mathbb{P}(T_{ji}\leq\tau-1).
\end{split}
\end{align}
Therefore, by \eqref{eq:self} and \eqref{eq:ineq} we have
\begin{align*}
\mathbb{P}(T_{ii}\leq\tau)&\geq p_{ii}+(1-p_{ii})\mathbb{P}(T_{ki}\leq\tau)\\
&\geq p_{ii}\mathbb{P}(T_{ki}\leq\tau)+(1-p_{ii})\mathbb{P}(T_{ki}\leq\tau)\\
&= \mathbb{P}(T_{ki}\leq\tau),
\end{align*}
which implies that attacking node $i$ when the surveillance agent is at node $k$ is a better strategy.
\end{proof}

\subsection{Dominant strategies for the surveillance agent}
In this subsection, we show that part of the optimal surveillance strategy can be determined readily when leaf nodes are present, where leaf nodes are nodes that have only one neighboring node.

\begin{lemma}[Dominant strategy on leaf nodes]\label{lemma:dominance_sur}
Given a strongly connected digraph $\mathcal{G}=(V,\mathcal{E})$ with $n\geq3$ nodes, if node $i\in V$ is a leaf node in $\mathcal{G}$ with node $j\in V$ as its only neighbor, then the optimal strategy $P^*$ satisfies
\begin{equation*}
P^*(i,i) = 0,\textup{ and } P^*(i,j) = 1.
\end{equation*}
\end{lemma}
\begin{proof}
Without loss of generality, suppose that node $1$ is a leaf node in $\mathcal{G}$ and node $2$ is its neighbor. Let $P$ be a strategy that is the same as $P^*$ except that $P(1,1)=p>0$ and $P(1,2)=1-p<1$. We prove that for all $i,j\in V$, the capture probability $\mathbb{P}(T_{ij}^*\leq\tau)$ for $P^*$ is greater than or equal to $\min_{\{i,j\in V\}}\mathbb{P}(T_{ij}\leq\tau)$ for $P$, which leads to $\min_{\{i,j\in V\}}\mathbb{P}(T_{ij}^*\leq\tau)\geq \min_{\{i,j\in V\}}\mathbb{P}(T_{ij}\leq\tau)$.

Since $P$ and $P^*$ differ by only the first row and node $1$ is a leaf node, we must have $\mathbb{P}(T_{i1}^*\leq \tau)=\mathbb{P}(T_{i1}\leq \tau)$ for all $i\in\{2,\dots,n\}$.  Moreover, by Lemma~\ref{coro:intruder_dominance}, the strategy $(1,1)$ is a dominated strategy for the intruder, thus $\min_{\{i,j\in V\}}\mathbb{P}(T_{ij}^*\leq\tau)$ and $\min_{\{i,j\in V\}}\mathbb{P}(T_{ij}\leq\tau)$ do not attain minimum at $(1,1)$.

We next prove that  $\mathbb{P}(T_{1j}^*\leq\tau)\geq\mathbb{P}(T_{1j}\leq\tau)$ for all $j\in\{2,\dots,n\}$ by induction. Let $d_{1j}$ be the length of the shortest paths from node $1$ to node $j$. The probabilities of these shortest paths are equal to the products of the edge probabilities along the paths, and for $P$ and $P^*$ they differ by a factor of $1-p$. Thus, we have that $\mathbb{P}(T_{1j}^*\leq d_{1j})>\mathbb{P}(T_{1j}\leq d_{1j})$. Suppose when $\tau\leq t$ for $t\geq d_{1j}$, we have  $\mathbb{P}(T_{1j}^*\leq \tau)>\mathbb{P}(T_{1j}\leq \tau)$ and let $\ell_{ij}^t$ be the set of paths from node $i$ to node $j$ that do not contain node $1$ and have length less than or equal to $t$ for $i\in\{2,\dots,n\}$, then for $\tau=t+1$,
\begin{align*}
\mathbb{P}(T_{1j}^*\leq  t+1)&=\mathbb{P}(T_{2j}^*\leq  t)\\
&=\sum_{t_1=1}^t\mathbb{P}(T_{21}^*=  t_1)\mathbb{P}(T_{1j}^*\leq  t-t_1)+\mathbb{P}(\ell_{2j}^t)\\
&\geq \sum_{t_1=1}^t\mathbb{P}(T_{21}=  t_1)\mathbb{P}(T_{1j}\leq  t-t_1)+\mathbb{P}(\ell_{2j}^t)\\
&=\mathbb{P}(T_{2j}\leq  t)\\
&\geq p\mathbb{P}(T_{1j}\leq  t)+(1-p)\mathbb{P}(T_{2j}\leq  t)\\
&=\mathbb{P}(T_{1j}\leq  t+1),
\end{align*}
where the first inequality follows from the induction hypothesis  and the second inequality follows from Lemma~\ref{lemma:dominance_int}.

Finally, for $i,j\in\{2,\dots,n\}$, we have
\begin{align*}
\mathbb{P}(T_{ij}^*\leq  \tau)&=\sum_{t=1}^{\tau}\mathbb{P}(T_{i1}^*=  t)\mathbb{P}(T_{1j}^*\leq  \tau-t)+\mathbb{P}(\ell_{ij}^{\tau})\\
&\geq \sum_{t=1}^{\tau}\mathbb{P}(T_{i1}=  t)\mathbb{P}(T_{1j}\leq  \tau-t)+\mathbb{P}(\ell_{ij}^{\tau})\\
&=\mathbb{P}(T_{ij}\leq  \tau).
\end{align*}

In summary, we have that $\min_{\{i,j\in V\}}\mathbb{P}(T_{ij}^*\leq\tau)\geq \min_{\{i,j\in V\}}\mathbb{P}(T_{ij}\leq\tau)$, which completes the proof. 
\end{proof}

\subsection{Optimal solution for star graphs}
In this subsection, we consider the star topology, which represents the abstraction of an environment where there is a corridor connecting multiple rooms. The optimal strategy for the surveillance agent is given in the following theorem.

\begin{theorem}[Optimal solution in star graph]\label{thm:star}
	Given a directed star graph $\mathcal{G}=(V,\mathcal{E})$ with $n\geq3$ nodes and node $1$ being the center, the optimal strategy $P^*$ for the surveillance agent is given by
	\begin{equation}\label{eq:staroptimal}
	P^*=\begin{bmatrix}
	0&\frac{1}{n-1}&\frac{1}{n-1}&\cdots&\frac{1}{n-1}\\
	1&0&0&\cdots&0\\
	1&0&0&\cdots&0\\
	\vdots&\vdots&\vdots&\cdots&\vdots\\
	1&0&0&\cdots&0
	\end{bmatrix}.
	\end{equation}
	Moreover, the value of the game $\mathbb{V}$ satisfies
	\begin{equation*}
	\mathbb{V}=\begin{cases}
	1-(1-\frac{1}{n-1})^{\frac{\tau-1}{2}},&\quad\textup{if }\tau\geq2\textup{ is odd},\\
	1-(1-\frac{1}{n-1})^{\frac{\tau}{2}},&\quad\textup{if }\tau\geq2\textup{ is even}.
	\end{cases}
	\end{equation*}
\end{theorem}
\begin{proof}
For the directed star graph $\mathcal{G}$, the Markov chain $P$ corresponding to $\mathcal{G}$ has the following general structure,
\begin{equation*}
P=\begin{bmatrix}
p_{11}&p_{12}&p_{13}&\cdots&p_{1n}\\
p_{21}&p_{22}&0&\cdots&0\\
p_{31}&0&p_{33}&\cdots&0\\
\vdots&\vdots&\vdots&\cdots&\vdots\\
p_{n1}&0&0&\cdots&p_{nn}
\end{bmatrix}.
\end{equation*}
Since node $2$ to node $n$ are leaf nodes, by Lemma~\ref{lemma:dominance_sur}, the optimal Markov chain does not have self loops at these nodes and we can reduce $P$ to
\begin{equation}\label{eq:reducedP}
P=\begin{bmatrix}
p_{11}&p_{12}&p_{13}&\cdots&p_{1n}\\
1&0&0&\cdots&0\\
1&0&0&\cdots&0\\
\vdots&\vdots&\vdots&\cdots&\vdots\\
1&0&0&\cdots&0
\end{bmatrix}.
\end{equation}
Note that the strategies $(1,j)$ are dominated for the intruder for all
$j\in\{2,\dots,n\}$ by Lemma~\ref{lemma:dominance_int}, and the capture
probabilities $\mathbb{P}(T_{i1}\leq \tau)=1$ for all $i\in\{1,\dots,n\}$
and $\tau\geq2$. Therefore, we only need to find a $P$ in~\eqref{eq:reducedP}
that maximizes $\min_{i,j\in\{2,\dots,n\}}\mathbb{P}(T_{ij}\leq\tau)$. We divide the rest of the proof into two parts. In the first part, we show that there is no self loop at the center node for the
optimal solution, i.e., $p_{11}=0$ in~\eqref{eq:reducedP}, and further reduce $P$; in the second part, we obtain the optimal solution.

\paragraph*{No self loop at center} We construct $P_1$ to be the same as $P$ in~\eqref{eq:reducedP} except for the first row where
\begin{equation*}
P_1(1,1)=0, ~P_1(1,j)=\frac{p_{1j}}{1-p_{11}}\textup{ for }j\in\{2,\dots,n\}.
\end{equation*}
We show that $P_1$ is a better strategy than $P$ in~\eqref{eq:reducedP} by induction on $\tau$, i.e.,  for all $i,j\in\{2,\dots,n\}$, the capture probabilities $\mathbb{P}(T_{ij}^1\leq\tau)$ for $P_1$ is greater than $\mathbb{P}(T_{ij}\leq\tau)$ of $P$. When $\tau=2$, we have $\mathbb{P}(T_{ij}^1\leq2)=\frac{p_{1j}}{1-p_{11}}>p_{1j}=\mathbb{P}(T_{ij}\leq2)$. Suppose when $\tau \leq t$, we have $\mathbb{P}(T_{ij}^1\leq\tau)>\mathbb{P}(T_{ij}\leq\tau)$, then for $\tau=t+1$,
\begin{align*}
\mathbb{P}(T_{ij}^1\leq t+1)&=\mathbb{P}(T_{1j}^1\leq t)\\
&=\frac{p_{1j}}{1-p_{11}}+\sum_{k\notin \{1,j\}}\frac{p_{1k}}{1-p_{11}}\mathbb{P}(T_{kj}^1\leq t-1)\\
&>\frac{p_{1j}}{1-p_{11}}+\sum_{k\notin \{1,j\}}\frac{p_{1k}}{1-p_{11}}\mathbb{P}(T_{kj}\leq t-1)\\
&=\frac{1}{1-p_{11}}(p_{1j}+\sum_{k\notin \{1,j\}}p_{1k}\mathbb{P}(T_{1j}\leq t-1))\\
&=\frac{1}{1-p_{11}}(\mathbb{P}(T_{1j}\leq t)-p_{11}\mathbb{P}(T_{1j}\leq t-1))\\
&\geq\frac{1}{1-p_{11}}(\mathbb{P}(T_{1j}\leq t)-p_{11}\mathbb{P}(T_{1j}\leq t))\\
&=\mathbb{P}(T_{1j}\leq t)\\
&=\mathbb{P}(T_{ij}\leq t+1),
\end{align*}
where the first inequality follows from the induction hypothesis and the second inequality follows from the fact that $\mathbb{P}(T_{1j}\leq t-1)\leq\mathbb{P}(T_{1j}\leq t)$. Therefore, we conclude that the optimal strategy does not have a self loop at the center node.

\paragraph*{Optimal solution}
Note that for any $\tau\geq1$ and $j\in\{2,\dots,n\}$, we have that
\begin{align*}
\mathbb{P}(T_{1j}\leq\tau)&=p_{1j}+\sum_{ k\notin\{1,j\}}p_{1k}\mathbb{P}(T_{kj}\leq\tau-1)\\
&=p_{1j}+\sum_{ k\notin\{1,j\}}p_{1k}\mathbb{P}(T_{1j}\leq\tau-2)\\
&=p_{1j}+(1-p_{1j})\mathbb{P}(T_{1j}\leq\tau-2),
\end{align*}
with the initial condition $\mathbb{P}(T_{1j}\leq1)=\mathbb{P}(T_{1j}\leq2)=p_{1j}$. Therefore, the capture probability $\mathbb{P}(T_{1j}\leq\tau)$ satisfies
\begin{equation}\label{eq:starcases}
\mathbb{P}(T_{1j}\leq\tau)=\begin{cases}
1-(1-p_{1j})^{\frac{\tau-1}{2}},\quad&\textup{if }\tau\geq2\textup{ is odd},\\
1-(1-p_{1j})^{\frac{\tau}{2}},\quad&\textup{if }\tau\geq2\textup{ is even}.
\end{cases}
\end{equation}
By~\eqref{eq:starcases}, we have for odd $\tau\geq2$, 
\begin{align*}
\min_{i,j\in\{2,\dots,n\}}\mathbb{P}(T_{ij}\leq\tau)&=\min_{j\in\{2,\dots,n\}}\mathbb{P}(T_{1j}\leq\tau-1)\\
&=\min_{j\in\{2,\dots,n\}}1-(1-p_{1j})^{\frac{\tau-1}{2}}\\
&=1-\max_{j\in\{2,\dots,n\}}(1-p_{1j})^{\frac{\tau-1}{2}}\\
&=1-(1-\min_{j\in\{2,\dots,n\}}p_{1j})^{\frac{\tau-1}{2}},
\end{align*}
which along with the fact that $\sum_{j=2}^np_{1j}=1$ implies that~\eqref{eq:staroptimal} is the optimal solution for the directed star graph.
\end{proof}

\subsection{Optimal solution for line graphs}
In this subsection, we derive the optimal surveillance strategy for directed line graphs. In an $n$-node line graph, Problem~\ref{prob:optimalstrategy} is interesting only when the attack duration $\tau$ satisfies $n-1\leq\tau\leq 2n-3$. If $\tau<n-1$, the omniscient intruder always succeeds by attacking an end node when the surveillance agent is at the other end; if $\tau>2n-3$, the surveillance agent who walks back and forth between two ends of the line graph (a deterministic sweeping) 
captures the omniscient intruder no matter how it attacks. Therefore, we  consider only cases when $n-1\leq \tau\leq 2n-3$. Note that a sweeping strategy fails as long as $\tau< 2n-3$, because the intruder could attack an end node immediately after the surveillance agent just leaves that node and it succeeds with probability $1$. We label the nodes in an $n$-node line graph successively from left to right by $(1,\dots,n)$. We need the following conjecture to establish our main result.

\begin{conjecture}[Uniqueness of the optimal strategy]\label{conj:uniqueness}
	Given a directed line graph $\mathcal{G}=(V,\mathcal{E})$ with $n\geq3$ nodes, the optimal solution to Problem~\ref{prob:optimalstrategy} is unique.
\end{conjecture}

We provide evidence for Conjecture~\ref{conj:uniqueness} in Remark~\ref{rk:uniqueness}.

\begin{theorem}[Optimal solution in line graph]
	Given a directed line graph $\mathcal{G}=(V,\mathcal{E})$ with $n\geq3$ nodes, if Conjecture~\ref{conj:uniqueness} holds true, then the optimal strategy $P^*$ is given by
	\begin{equation}\label{eq:lineoptimal}
	P^*=\begin{bmatrix}
	0&1&0&\cdots&0\\
	0.5&0&0.5&\cdots&0\\
	\vdots&\ddots&\ddots&\ddots&\vdots\\
	0&\cdots&0.5&0&0.5\\
	0&\cdots&0&1&0\\
	\end{bmatrix}.
	\end{equation}
\end{theorem}
\begin{proof}
We divide the proof into three parts. In the first part, we show that in the line graph, the optimal strategy for the intruder is to attack one end of the graph when the surveillance agent is at the other, in which case the objective function in Problem~\ref{prob:optimalstrategy} becomes $\min\{\mathbb{P}(T_{1n}\leq \tau),\mathbb{P}(T_{n1}\leq\tau)\}$. In the second part, we show that there are no self loops at any locations in the optimal surveillance strategy. In the last part, we obtain the optimal strategy by using Conjecture~\ref{conj:uniqueness} and a symmetry argument.
	
\paragraph*{Attack the end nodes} For $i,j\in V$ and $i< j$, by Lemma \ref{lemma:dominance_int}, we know that $\mathbb{P}(T_{ij}\leq\tau)\geq\mathbb{P}(T_{1j}\leq\tau)\geq\mathbb{P}(T_{1n}\leq\tau)$; on the other hand, for $i>j$, we have $\mathbb{P}(T_{ij}\leq\tau)\geq\mathbb{P}(T_{nj}\leq\tau)\geq\mathbb{P}(T_{n1}\leq\tau)$. Therefore, attacking any location in the middle while the surveillance agent is at another location in the middle is not optimal for the intruder. Since node $1$ and $n$ are leaf nodes, by Lemma~\ref{coro:intruder_dominance}, it is not optimal for the intruder to attack node $1$ or $n$ immediately after the surveillance agent leaves that node. Next, we show that attacking any node $i\in\{2,\dots,n-1\}$ immediately when the surveillance agent leaves node $i$ is dominated by attacking an end node when the surveillance agent is visiting the other. For $i\in\{2,\dots,n-1\}$,
\begin{align*}
\mathbb{P}(T_{ii}\leq \tau)&=p_{ii} + p_{i,i+1}\mathbb{P}(T_{i+1,i}\leq \tau-1) \\
&\quad+ p_{i,i-1}\mathbb{P}(T_{i-1,i}\leq \tau-1)\\
&\geq p_{ii} + p_{i,i+1}\mathbb{P}(T_{n1}\leq \tau) + p_{i,i-1}\mathbb{P}(T_{1n}\leq \tau)\\
&\geq \min\{1,\mathbb{P}(T_{n1}\leq \tau),\mathbb{P}(T_{1n}\leq \tau)\}\\
&= \min\{\mathbb{P}(T_{n1}\leq \tau),\mathbb{P}(T_{1n}\leq \tau)\},
\end{align*}
where the first inequality follows from the facts that 
\begin{align*}
&\quad\mathbb{P}(T_{n1}\leq \tau)\\&=\sum_{t_1,t_2}\mathbb{P}(T_{n,i+1}= t_1)\mathbb{P}(T_{i+1,i}\leq \tau-t_1-t_2)\mathbb{P}(T_{i1}=t_2)\\
&\leq\mathbb{P}(T_{i+1,i}\leq \tau-1),
\end{align*}
and
\begin{align*}
&\quad\mathbb{P}(T_{1n}\leq \tau)\\&=\sum_{t_1,t_2}\mathbb{P}(T_{1,i}= t_1)\mathbb{P}(T_{i,i+1}\leq \tau-t_1-t_2)\mathbb{P}(T_{i+1,n}=t_2)\\
&\leq\mathbb{P}(T_{i,i+1}\leq \tau-1).
\end{align*}
Therefore, the best strategy for the omniscient intruder is to attack an end node when the surveillance agent is at the other. Problem~\ref{prob:optimalstrategy} becomes $\max_{P}\min\{\mathbb{P}(T_{1n}\leq \tau),\mathbb{P}(T_{n1}\leq\tau)\}$.
	
\paragraph*{No self loop at any location} Fist, by Lemma~\ref{lemma:dominance_sur}, since nodes $1$ and $n$ are leaf nodes, the optimal strategy $P^*$ for the surveillance agent must satisfy $P^*(1,1)=0$, $P^*(1,2)=1$, $P^*(n,n)=0$ and $P^*(n,n-1)=1$. Next, we focus on $i\in\{2,\dots,n-1\}$. Let $P$ be any Markov chain strategy corresponding to the line graph with $p_{ii}>0$, and $P_1$ is the same as $P$ except for the $i$-th row where
\begin{equation*}
P_1(i,i)=0, ~P_1(i,i+1)=\frac{p_{i,i+1}}{1-p_{ii}},~P_1(i,i-1)=\frac{p_{i,i-1}}{1-p_{ii}}.
\end{equation*}
Note that
\begin{align}\label{eq:linerecursion}
\begin{split}
\mathbb{P}(T_{1n}\leq\tau)&=\sum_{t_1=i-1}^{\tau}\mathbb{P}(T_{1i}= t_1)\mathbb{P}(T_{in}\leq \tau- t_1),\\
\mathbb{P}(T_{n1}\leq\tau)&=\sum_{t_1=n-i}^{\tau}\mathbb{P}(T_{ni}=t_1)\mathbb{P}(T_{i1}\leq \tau- t_1).
\end{split}
\end{align}
Since $P$ and $P_1$  differ only by row $i$, their first hitting time probabilities satisfy $\mathbb{P}(T_{1i}= t_1)=\mathbb{P}(T^1_{1i}= t_1)$ and  $\mathbb{P}(T_{ni}= t_1)=\mathbb{P}(T^1_{ni}= t_1)$ for all $t_1\geq1$. We first prove that $\mathbb{P}(T_{in}\leq \tau)\leq\mathbb{P}(T^1_{in}\leq \tau)$ for all $\tau$ by induction. When $\tau=n-i$, since $P_1(i,i+1)>p_{i,i+1}$, we have $\mathbb{P}(T_{in}\leq n-i)\leq\mathbb{P}(T^1_{in}\leq n-i)$. Suppose for all $\tau\leq  t$, we have $\mathbb{P}(T_{in}\leq  t)\leq\mathbb{P}(T^1_{in}\leq  t)$. Then, when $\tau= t+1$,

\begin{align*}
&\quad\mathbb{P}(T^1_{in}\leq  t+1)\\
&=\frac{p_{i,i-1}}{1-p_{ii}}\mathbb{P}(T^1_{i-1,n}\leq  t)+\frac{p_{i,i+1}}{1-p_{ii}}\mathbb{P}(T^1_{i+1,n}\leq  t),\\
&\geq \frac{p_{i,i-1}}{1-p_{ii}}\mathbb{P}(T_{i-1,n}\leq  t)+\frac{p_{i,i+1}}{1-p_{ii}}\mathbb{P}(T_{i+1,n}\leq  t)\\
&=\frac{1}{1-p_{ii}}(\mathbb{P}(T_{in}\leq  t+1)-p_{ii}\mathbb{P}(T_{in}\leq  t))\\
&\geq\frac{1}{1-p_{ii}}(\mathbb{P}(T_{in}\leq  t+1)-p_{ii}\mathbb{P}(T_{in}\leq  t+1))\\
&=\mathbb{P}(T_{in}\leq  t+1),
\end{align*}
where the first inequality follows from the hypothesis induction. A similar proof by induction shows that  $\mathbb{P}(T_{ni}\leq \tau)\leq\mathbb{P}(T^1_{ni}\leq \tau)$ for all $\tau$. Then, by~\eqref{eq:linerecursion}, we have that $\mathbb{P}(T_{n1}\leq \tau)\leq\mathbb{P}(T^1_{n1}\leq \tau)$ and $\mathbb{P}(T_{1n}\leq \tau)\leq\mathbb{P}(T^1_{1n}\leq \tau)$ and therefore $P_1$ is a better strategy than $P$. In summary, we have that the optimal surveillance strategy does not have self loop at any location.

\paragraph*{Optimal solution} By the first two parts, we conclude that the optimal strategy for the surveillance agent has the following general structure,
\begin{equation}\label{eq:generalP}
P=\begin{bmatrix}
0&1&0&\cdots&0\\
1-x_1&0&x_1&\cdots&0\\
\vdots&\ddots&\ddots&\ddots&\vdots\\
0&\cdots&1-x_{n-2}&0&x_{n-2}\\
0&\cdots&0&1&0\\
\end{bmatrix},
\end{equation}
and thus the objective function in Problem~\ref{prob:optimalstrategy} can be parameterized by $0<x_1<1,\dots,0<x_{n-2}<1$. Let
\begin{equation*}
f(x_1,\dots,x_{n-2})=\min_{x_1,\dots,x_{n-2}}\{\mathbb{P}(T_{1n}\leq \tau),\mathbb{P}(T_{n1}\leq\tau)\}.
\end{equation*}
For the function $f$, following the proof in Appendix~\ref{appdix:proof1}, we have 
\begin{equation*}
f(x_1,\dots,x_{n-2})=f(1-x_1,\dots,1-x_{n-2}).
\end{equation*}
If Conjecture~\ref{conj:uniqueness} holds, then at the optimal solution, we must have $x^*_i=1-x^*_{i}$ for all $i\in\{1,\dots,n-2\}$,
which implies $x_{i}^*=\frac{1}{2}$. Therefore, the optimal solution is given by~\eqref{eq:lineoptimal}. 
\end{proof}



We prove a necessary condition for a strategy to be optimal in line graphs in Lemma~\ref{lemma:linenecessary} in Appendix~\ref{appendix:linenecessary}, which says that the optimal strategy must satisfy $\mathbb{P}(T^*_{1n}\leq\tau)=\mathbb{P}(T^*_{1n}\leq\tau)$. We provide evidence for Conjecture~\ref{conj:uniqueness} in the following remark.

\begin{remark}[Evidence for conjecture~\ref{conj:uniqueness}]\label{rk:uniqueness}
We first consider two tractable cases: $n=3$ and $\tau=n-1$ or $\tau=n$. When $n=3$, the line graph is also a star graph, and by Theorem~\ref{thm:star}, we have that the optimal solution in unique. The case $\tau=n$ is the same as that of $\tau=n-1$ since the optimal Markov chain in the form~\eqref{eq:generalP} is periodic. For $\tau=n-1$, we have
\begin{align*}
&\mathbb{P}(T_{1n}\leq n-1)=p_{23}p_{34}\cdots p_{n-1,n},\\
&\mathbb{P}(T_{n1}\leq n-1)=p_{n-1,n-2}p_{n-2,n-3}\cdots p_{21},
\end{align*}
and note that
\begin{align}\label{eq:lineproduct}
\begin{split}
&\quad\mathbb{P}(T_{1n}\leq n-1)\mathbb{P}(T_{n1}\leq n-1)\\
&=(p_{21}p_{23})\cdots (p_{n-1,n-2}p_{n-1,n})\\
&\leq \big(\frac{p_{21}+p_{23}}{2}\big)^2\cdots\big(\frac{p_{n-1,n-2}+p_{n-1,n}}{2}\big)^2=\frac{1}{2^{2n-4}}.
\end{split}
\end{align}
There is a unique strategy that achieves the upper bound in~\eqref{eq:lineproduct}: $p_{i,i+1}=p_{i,i-1}=\frac{1}{2}$ for all $i\in\{2,\dots,n-1\}$. Moreover, this strategy satisfies the necessary condition in Lemma~\ref{lemma:linenecessary}. Therefore, the optimal solution is unique.

As a second justification, following the Monte Carlo probability estimation method in~\cite[Remark V.1]{GN-FB:07z}, we randomly pick $27000$ different Markov chains with the structure~\eqref{eq:generalP} for $n=5$, $\tau=8$ and $n=6$, $\tau=8$. In all these cases, no chain has been found to have a better or same value as that of~\eqref{eq:lineoptimal}. Therefore, we have $99\%$ confidence that the probability of~\eqref{eq:lineoptimal} being optimal is at least $0.99$ for these cases.
\end{remark}

\section{Conclusion}\label{sec:conclusion}
In this paper, we studied a Stackelberg game formulation for the robotic
surveillance problem, where the surveillance agent defends against an
omniscient intruder who decides where and when to attack. We derived an
upper bound on the performance of the surveillance agent and provided
provably suboptimal solution in the complete graph. We derived dominant
strategies and leveraged them to obtain optimal strategies for the star and
the line topology. For future works, we will consider arbitrary graph
topology and heterogeneous attack duration.

\section{Acknowledgement} The authors would like to thank Prof. Silvano Martello and Prof. Michele Monaci for numerous inspiring discussions.

\begin{appendices}
\section{Proof of a symmetry property}\label{appdix:proof1}
	Let $P_1$ be a Markov chain of the form~\eqref{eq:generalP} and $P_2$ be 
	\begin{equation*}
	\begin{bmatrix}
	0&1&0&\cdots&0\\
	x_1&0&1-x_1&\cdots&0\\
	\vdots&\ddots&\ddots&\ddots&\vdots\\
	0&\cdots&x_{n-2}&0&1-x_{n-2}\\
	0&\cdots&0&1&0\\
	\end{bmatrix},
	\end{equation*}
	which is, in terms of the objective function, equivalent to
	\begin{equation}\label{eq:P22}
	P_2=\begin{bmatrix}
	0&1&0&\cdots&0\\
	1-x_{n-2}&0&x_{n-2}&\cdots&0\\
	\vdots&\ddots&\ddots&\ddots&\vdots\\
	0&\cdots&1-x_{1}&0&x_{1}\\
	0&\cdots&0&1&0\\
	\end{bmatrix}.
	\end{equation}
	We work with $P_2$ in~\eqref{eq:P22} for the rest of the proof.
	Let $\mathbb{P}(T_{1n}^1\leq\tau)$ and $\mathbb{P}(T_{1n}^2\leq\tau)$ be the capture probabilities for $P_1$ and $P_2$, respectively. We show that $\mathbb{P}(T_{1n}^1\leq\tau)=\mathbb{P}(T_{1n}^2\leq\tau)$, and a similar proof strategy works for the claim $\mathbb{P}(T_{n1}^1\leq\tau)=\mathbb{P}(T_{n1}^2\leq\tau)$. From~\eqref{eq:RecursiveInVec}, we have that
	\begin{align*}
	\mathbb{P}(T_{1n}^1\leq\tau)&=\mathbb{e}_1^\top\sum_{t=1}^\tau (P_1 - P_1\mathbb{e}_n\mathbb{e}_n^\top)^{t-1}P\mathbb{e}_n\\
	&=\mathbb{e}_1^\top (I_n-(P_1 - P_1\mathbb{e}_n\mathbb{e}_n^\top)^{\tau})\\
	&\quad\cdot(I_n-(P_1 - P_1\mathbb{e}_n\mathbb{e}_n^\top))^{-1}P\mathbb{e}_n\\
	&=\mathbb{e}_1^\top(I_{n-1}-A^\tau)(I_{n-1}-A)^{-1}b,
	\end{align*}
	where $b=x_{n-2}\mathbb{e}_{n-1}$ and 
	\begin{equation*}
	A=\begin{bmatrix}
	0&1&0&\cdots&0\\
	1-x_1&0&x_1&\cdots&0\\
	\vdots&\ddots&\ddots&\ddots&\vdots\\
	0&\cdots&1-x_{n-3}&0&x_{n-3}\\
	0&\cdots&0&1-x_{n-2}&0\\
	\end{bmatrix}.
	\end{equation*}
	Note that $(I_{n-1}-A)\mathbb{1}_{n-1}=b$. Therefore, we have
	\begin{equation}\label{eq:captureP1}
	\mathbb{P}(T_{1n}^1\leq\tau)=\mathbb{e}_1^\top(I_{n-1}-A^{\tau})\mathbb{1}_{n-1}.
	\end{equation}
	Similarly, for $P_2$, we have
	\begin{equation}\label{eq:captureP2}
	\mathbb{P}(T_{1n}^2\leq\tau)=\mathbb{e}_1^\top(I_{n-1}-B^{\tau})\mathbb{1}_{n-1},
	\end{equation}
	where
	\begin{equation*}
	B=\begin{bmatrix}
	0&1&0&\cdots&0\\
	1-x_{n-2}&0&x_{n-2}&\cdots&0\\
	\vdots&\ddots&\ddots&\ddots&\vdots\\
	0&\cdots&1-x_{2}&0&x_{2}\\
	0&\cdots&0&1-x_{1}&0\\
	\end{bmatrix}.
	\end{equation*}
	In order to show $\mathbb{P}(T_{1n}^1\leq\tau)=\mathbb{P}(T_{1n}^2\leq\tau)$, by~\eqref{eq:captureP1} and \eqref{eq:captureP2}, we only need to prove that for all $\tau\geq1$,
	\begin{equation}\label{eq:powerequal}
	\mathbb{e}_1^\top A^{\tau}\mathbb{1}_{n-1}=\mathbb{e}_1^\top B^{\tau}\mathbb{1}_{n-1}.
	\end{equation}
	First, note that for all $\tau\leq n-2$, we have
	\begin{equation}\label{eq:powerequalsmall}
	\mathbb{e}_1^\top A^{\tau}\mathbb{1}_{n-1}=\mathbb{e}_1^\top B^{\tau}\mathbb{1}_{n-1}=1.
	\end{equation}
	We next show that $A$ and $B$ have the same characteristic
        polynomials, and then by the Cayley-Hamilton and
        \eqref{eq:powerequalsmall}, we will
        have~\eqref{eq:powerequal}. Since both $A$ and $B$ are tridiagonal
        matrices, the characteristic polynomials of $A$ and $B$ can be
        generated as follows \cite[equation~(2.3)]{MEAE:04}. Note that $A$
        and $B$ are of order $n-1$. Let
        $g_0^{n-1}(\lambda)=h_0^{n-1}(\lambda)=1$,
        $g_1^{n-1}(\lambda)=h_1^{n-1}(\lambda)=\lambda$, where the
        superscript indicates the order of the matrices, and for
        $k=2,\dots,n-1$,
	\begin{align}\label{eq:ghrecursive}
	\begin{split}
	g_k^{n-1}(\lambda)&=\lambda g_{k-1}^{n-1}(\lambda) - x_{k-2}(1-x_{k-1})g_{k-2}^{n-1}(\lambda),\\
	h_k^{n-1}(\lambda)&=\lambda h_{k-1}^{n-1}(\lambda) - x_{k}(1-x_{k-1})h_{k-2}^{n-1}(\lambda),
	\end{split}
	\end{align}
	where $x_0=x_{n-1}=1$ and we obtain the recurrence $h_{k}^{n-1}$ for $B$ starting from the bottom right of $B$ matrix. Then $g_{n-1}^{n-1}(\lambda)$ and $h_{n-1}^{n-1}(\lambda)$ are the characteristic polynomials for $A$ and $B$, respectively. Moreover, notice that 
	\begin{align*}
	g_{n}^{n}&=\begin{vmatrix}
	\lambda I_{n-1} -A &-x_{n-2}\mathbb{e}_{n-1}\\
	-(1-x_{n-1})\mathbb{e}_{n-1}^\top&\lambda
	\end{vmatrix}\\
	&=\lambda g_{n-1}^{n-1}-x_{n-2}(1-x_{n-1})g_{n-2}^{n-1},
	\end{align*}
	and $g_{n-1}^n=g_{n-1}^{n-1}$. Thus,
	\begin{align}\label{eq:gtog}
	\begin{split}
	\begin{bmatrix}
	g_{n}^{n}(\lambda)\\
	g_{n-1}^{n}(\lambda)
	\end{bmatrix}&=
	\begin{bmatrix}
	\lambda&- x_{n-2}(1-x_{n-1})\\
	1&0
	\end{bmatrix}\begin{bmatrix}
	g_{n-1}^{n-1}(\lambda)\\
	g_{n-2}^{n-1}(\lambda)
	\end{bmatrix}	
	\end{split}.
	\end{align}
	At the same time, note that
	\begin{align}\label{eq:hrecursive}
	\begin{split}
	h_{n-1}^{n}&=|\lambda I_{n-1}-B+(1-x_{n-1})\mathbb{e}_{1}\mathbb{e}_{2}^\top|\\
	&=\lambda h_{n-2}^{n-1}-(1-x_{n-2})x_{n-1}h_{n-3}^{n-1},\\
	h_{n-2}^{n}&=h^{n-1}_{n-2}.	
	\end{split}
	\end{align}
	Thus,
	\begin{align}\label{eq:htoh}	
	\begin{split}
	\begin{bmatrix}
	h_{n}^{n}(\lambda)\\
	h_{n-1}^{n}(\lambda)
	\end{bmatrix}&=
	\begin{bmatrix}
	\lambda&-(1-x_{n-1})\\
	1&0
	\end{bmatrix}\begin{bmatrix}
	h_{n-1}^{n}(\lambda)\\
	h_{n-2}^{n}(\lambda)
	\end{bmatrix}\\
	&=
	\begin{bmatrix}
	\lambda&-(1-x_{n-1})\\
	1&0
	\end{bmatrix}\\
	&\quad\cdot\begin{bmatrix}
	\lambda&-x_{n-1}(1-x_{n-2})\\
	1&0
	\end{bmatrix}\begin{bmatrix}
	h_{n-2}^{n-1}(\lambda)\\
	h_{n-3}^{n-1}(\lambda)
	\end{bmatrix}\\
	&=
	\begin{bmatrix}
	\lambda&-(1-x_{n-1})\\
	1&0
	\end{bmatrix}\begin{bmatrix}
	\lambda&-x_{n-1}(1-x_{n-2})\\
	1&0
	\end{bmatrix}\\
	&\quad\cdot\begin{bmatrix}
	\lambda&-(1-x_{n-2})\\
	1&0
	\end{bmatrix}^{-1}\begin{bmatrix}
	h_{n-1}^{n-1}(\lambda)\\
	h_{n-1}^{n-2}(\lambda)
	\end{bmatrix}\\
	&=\begin{bmatrix}
	\lambda x_{n-1}&(\lambda^2-1)(1-x_{n-1})\\
	x_{n-1}&\lambda(1-x_{n-1})
	\end{bmatrix}\begin{bmatrix}
	h_{n-1}^{n-1}(\lambda)\\
	h_{n-2}^{n-1}(\lambda)
	\end{bmatrix},
	\end{split}
	\end{align}
	where the second and third equalities follow from~\eqref{eq:hrecursive} and~\eqref{eq:ghrecursive}, respectively.
In the following, we prove that
	\begin{align}\label{eq:byinductionmain}
	\begin{split}
	&g_{n-1}^{n-1}(\lambda)=h_{n-1}^{n-1}(\lambda),\\
	&\lambda g_{n-1}^{n-1}(\lambda) = x_{n-2}g_{n-2}^{n-1}(\lambda) + (\lambda^2-1)h_{n-2}^{n-1},	
	\end{split}
	\end{align}
by induction on $n$. When $n=3$, by~\eqref{eq:ghrecursive}, we have
	\begin{align*}
&g_{2}^{2}(\lambda)=\lambda g_{1}^{2}(\lambda) - x_{0}(1-x_{1})g_{0}^{2}(\lambda)=
\lambda^2-(1-x_1),\\
&h_2^{2}(\lambda)=\lambda h_{1}^{2}(\lambda) - x_{2}(1-x_{1})h_{0}^{2}(\lambda)=\lambda^2-(1-x_1),
%
%
%
%
\end{align*}
and
\begin{align*}
\lambda g_{2}^{2}(\lambda)=\lambda^3-\lambda(1-x_1)=x_{1}g_{1}^{2}(\lambda) + (\lambda^2-1)h_{1}^{2}.
\end{align*}
Suppose~\eqref{eq:byinductionmain} holds for $n=k$, then when $n=k+1$,
\begin{align*}
	&\quad g_{k}^{k}(\lambda)-h_{k}^{k}(\lambda)\\
	&=\lambda g_{k-1}^{k-1}(\lambda)- x_{k-2}(1-x_{k-1})g_{k-2}^{k-1}(\lambda)\\
	&\quad-\lambda x_{k-1} h_{k-1}^{k-1}(\lambda)-(\lambda^2-1)(1-x_{k-1})h_{k-2}^{k-1}(\lambda)\\
	&=x_{k-2}x_{k-1}g_{k-2}^{k-1}(\lambda)-\lambda x_{k-1} h_{k-1}^{k-1}(\lambda)\\
	&\quad+(\lambda^2-1)x_{k-1}h_{k-2}^{k-1}(\lambda)\\
	&=\lambda x_{k-1} g_{k-1}^{k-1}(\lambda)-\lambda x_{k-1} h_{k-1}^{k-1}(\lambda)=0,
	\end{align*}
	where the first equality follows from~\eqref{eq:gtog} and~\eqref{eq:htoh}, the second and third follow from induction hypothesis. Moreover,
	\begin{align*}
	&\quad\lambda g_{k}^{k}(\lambda) - x_{k-1}g_{k-1}^{k}(\lambda) - (\lambda^2-1)h_{k-1}^{k}\\
	&=(\lambda^2-x_{k-1})g_{k-1}^{k-1}(\lambda)-\lambda x_{k-2}(1-x_{k-1})g_{k-2}^{k-1}(\lambda)\\
	&\quad-(\lambda^2-1)(x_{k-1}h_{k-1}^{k-1}(\lambda)+\lambda(1-x_{k-1})h_{k-2}^{k-1}(\lambda))\\
	&=(1-x_{k-1})\lambda^2g_{k-1}^{k-1}(\lambda)-\lambda^2(1-x_{k-1})g_{k-1}^{k-1}(\lambda)=0,
	\end{align*}
	where the first equality follows from~\eqref{eq:gtog} and~\eqref{eq:htoh}, and the second follows from the induction hypothesis. The proof is completed.

\section{Necessary optimality condition in line graphs}\label{appendix:linenecessary}
\begin{lemma}[Necessary optimality condition in line graph]\label{lemma:linenecessary}
	Given a line graph $\mathcal{G}=(V,\mathcal{E})$ with $n\geq3$ nodes, if the Markov chain strategy $P^*$ is optimal for the surveillance agent, then it must satisfy 
	$\mathbb{P}(T^*_{1n}\leq\tau)=\mathbb{P}(T^*_{n1}\leq\tau)$.
\end{lemma}
\begin{proof}
	We divide the proof into two parts. In the first part, we show a monotonicity property of the hitting times $\mathbb{P}(T_{1n}\leq\tau)$ and $\mathbb{P}(T_{n1}\leq\tau)$. In the second part, we obtain the necessary condition.

\paragraph*{Monotonicity of hitting times} We claim that $\mathbb{P}(T_{1n}\leq \tau)$ is monotonically increasing (decreasing) with $p_{i,i+1}$ ($p_{i,i-1}$), and $\mathbb{P}(T_{n1}\leq \tau)$ is monotonically decreasing (increasing) with $p_{i,i+1}$ ($p_{i,i-1}$). We prove that $\mathbb{P}(T_{1n}\leq \tau)$ is monotonically increasing with $p_{i,i+1}$, and a similar proof works for the other cases. For any $i\in\{2,\dots,n-1\}$, let $P^\epsilon$ be a Markov chain that is the same as $P$ except that $P^{\epsilon}(i,i+1)=P(i,i+1)+\epsilon$ and $P^{\epsilon}(i,i-1)=P(i,i-1)-\epsilon$, where $\epsilon>0$ is small enough such that $P^\epsilon$ remains irreducible. We first show that for $i\in\{2,\dots,n-1\}$, we have $\mathbb{P}(T^\epsilon_{in}\leq \tau)\geq \mathbb{P}(T_{in}\leq \tau)$ by induction. When $\tau=n-i$,
	\begin{align*}
	\mathbb{P}(T^\epsilon_{in}\leq n-i)&=(\epsilon+p_{i,i+1})p_{i+1,i+2}\cdots p_{n-1,n}\\
	&>p_{i,i+1}p_{i+1,i+2}\cdots p_{n-1,n}\\
	&= \mathbb{P}(T_{in}\leq n-i).
	\end{align*}
	Suppose $\mathbb{P}(T^\epsilon_{in}\leq \tau)\geq \mathbb{P}(T^\epsilon_{in}\leq \tau)$ holds for $\tau\leq  t$. For $i\in\{2,\dots,n-2\}$, let $\ell_{i+1,n}^t$ be the set of paths from node $i+1$ to node $n$ that do not contain node $i$ and have length less than or equal to $t$. Then, when $\tau=t+1$, for $i\in\{2,\dots,n-2\}$,
	\begin{align}\label{eq:linefromiton}
	\begin{split}
	&\mathbb{P}(T^\epsilon_{in}\leq  t+1)\\
	&=(p_{i,i-1}-\epsilon)\mathbb{P}(T^\epsilon_{i-1,n}\leq  t)+ p_{i,i}\mathbb{P}(T^\epsilon_{in}\leq  t) \\
	&\quad+ (p_{i,i+1}+\epsilon)\mathbb{P}(T^\epsilon_{i+1,n}\leq  t)\\
	&=(p_{i,i-1}-\epsilon)\sum_{t_1=1}^{ t}\mathbb{P}(T^\epsilon_{i-1,i}=t_1)\mathbb{P}(T^\epsilon_{in}\leq  t-t_1)\\
	&\quad+p_{i,i}\mathbb{P}(T^\epsilon_{in}\leq  t)+ (p_{i,i+1}+\epsilon)\mathbb{P}(\ell_{i+1,n}^t)\\
	&\quad+(p_{i,i+1}+\epsilon)\sum_{t_1=1}^ t\mathbb{P}(T^\epsilon_{i+1,i}=t_1)\mathbb{P}(T^\epsilon_{in}\leq  t-t_1)\\
	&\geq (p_{i,i-1}-\epsilon)\mathbb{P}(T_{i-1,n}\leq  t)+ p_{i,i}\mathbb{P}(T_{in}\leq  t) \\
	&\quad+ (p_{i,i+1}+\epsilon)\mathbb{P}(T_{i+1,n}\leq  t)\\
	&\geq\mathbb{P}(T_{in}\leq  t+1),
	\end{split}
	\end{align}
	where the first inequality follows from the induction hypothesis and the second inequality follows from the fact that $\mathbb{P}(T_{i+1,n}\leq  t)\geq\mathbb{P}(T_{i-1,n}\leq t)$ by Lemma~\ref{lemma:dominance_int}. Moreover, when $i=n-1$, we equate $\mathbb{P}(T_{i+1,n}^\epsilon\leq t)=1$ in the third line of~\eqref{eq:linefromiton}, and a similar argument follows. Finally, we have
	\begin{align*}
	\mathbb{P}(T^\epsilon_{1n}\leq \tau)&=\sum_{t_1=1}^{\tau}\mathbb{P}(T^\epsilon_{1i}=t_1)\mathbb{P}(T^\epsilon_{in}\leq \tau-t_1)\\
	&\geq\sum_{t_1=1}^{\tau}\mathbb{P}(T_{1i}=t_1)\mathbb{P}(T_{in}\leq \tau-t_1)\\
	&= \mathbb{P}(T_{1n}\leq \tau),
	\end{align*}
	which completes the proof.

\paragraph*{Necessary condition for optimality} We prove by contradiction. Without loss of generality, suppose $P^*$ is optimal and  $\mathbb{P}(T^*_{1n}\leq\tau)<\mathbb{P}(T^*_{n1}\leq \tau)$. By the monotonicity properties in the first part, for $i\in\{2,n-1\}$, if we increase $p_{i,i+1}^*$ and decrease $p_{i,i-1}^*$, then $\mathbb{P}(T^*_{1n}\leq\tau)$ increases and $\mathbb{P}(T^*_{n1}\leq \tau)$ decreases continuously, which leads to an increase in the objective function  $\min\{\mathbb{P}(T^*_{1n}\leq\tau),\mathbb{P}(T^*_{n1}\leq \tau)\}$. Therefore, the strategy $P^*$ is not the optimal, which is a contradiction. 
\end{proof}
\end{appendices}

\bibliographystyle{plainurl+ISBN}
\bibliography{alias,Main,FB}

\end{document}